\documentclass[11pt,a4paper]{article}
\usepackage[utf8]{inputenc}
\usepackage[T1]{fontenc}
\usepackage{amsmath,amssymb,amsthm}
\usepackage[margin=2.5cm]{geometry}
\usepackage{booktabs}
\usepackage{xcolor}
\usepackage[colorlinks=true,linkcolor=blue!45!black,citecolor=blue!45!black,
            urlcolor=blue!45!black]{hyperref}
\theoremstyle{plain}
\newtheorem{teo}{Theorem}[section]
\newtheorem{prop}[teo]{Proposition}
\newtheorem{lema}[teo]{Lemma}
\newtheorem{cor}[teo]{Corollary}
\theoremstyle{definition}
\newtheorem{defi}[teo]{Definition}
\newtheorem{ex}[teo]{Example}
\theoremstyle{remark}
\newtheorem{obs}[teo]{Remark}

\newcommand{\X}{\mathcal{X}}
\newcommand{\W}{\mathcal{W}}
\newcommand{\R}{\mathbb{R}}
\newcommand{\C}{\mathbb{C}}

\title{Rational $S$-functions for second order ODEs\\
with non-Liouvillian first integrals}

\author{L.\,G.\,S.\ Duarte\thanks{Universidade do Estado do Rio de
Janeiro, Instituto de Física, Depto.\ de Física Teórica.
\texttt{lgsduarte@gmail.com}}
\and
L.\,A.\,C.\,P.\ da Mota\thanks{Universidade do Estado do Rio de Janeiro,
Instituto de Física, Depto.\ de Física Teórica.
\texttt{lacpdamota@uerj.br}}}
\date{October 2026}

\begin{document}
\maketitle

\begin{abstract}
The $S$-function method for rational second order ODEs
$y''=\varphi(x,y,y')$ uses $S=I_y/I_z$, attached to a first integral $I$,
which satisfies a Riccati equation along the flow; it was designed for
Liouvillian $I$. We study the case of rational $S$ with non-Liouvillian $I$.
From any rational $g(x,y,z)$ with $g_z\not\equiv0$ and any linear equation
$A_2w''+A_1w'+A_0w=0$ with polynomial coefficients, the substitution
$w'/w=g(x,y,y')$ yields a rational 2ODE whose first integral is built from
two independent solutions, and $S=g_y/g_z$: the Wronskian cancels, so $S$
does not see the linear equation, and one rational $S$ is shared by
equations with and without Liouvillian first integrals. Conversely, given
$(\varphi,S)$, membership in this class reduces to a rational first integral
of an explicit planar vector field plus a Möbius-invariant Riccati test in
$x$, which recovers the linear equation. A transfer theorem shows that $I$ is
Liouvillian if and only if this equation has a Liouvillian solution, so
Liouvillianity is decidable on the class by Kovacic's algorithm: a
three-variable problem with no a priori degree bound is reduced to two
classical ones. On nine pull-back equations of degree up to $15$, given
without information on their construction, the procedure recovers the Airy
equation and the Bessel equation of order zero.
\end{abstract}


\section{Introduction}\label{sec:intro}

\subsection{The $S$-function}

Let
\begin{equation}\label{eq:2ode}
y''=\varphi(x,y,z),\qquad z:=y',\qquad \varphi=\frac MN\in\R(x,y,z),
\end{equation}
with $M,N$ coprime polynomials, and let
\begin{equation}\label{eq:campo}
\X:=\partial_x+z\,\partial_y+\varphi\,\partial_z
\end{equation}
be the associated vector field. If $I(x,y,z)$ is a first integral of $\X$,
the \emph{$S$-function associated with $\X$ through $I$} is
\begin{equation}\label{eq:S}
S:=\frac{I_y}{I_z},
\end{equation}
introduced in~\cite{S2017} as the basis of a method --- implemented in the
Maple package \texttt{InSyDE} --- for 2ODEs on which both the Darboux and the
Lie approaches stall. Differentiating~\eqref{eq:S} along $\X$ gives the
equation on which everything here depends,
\begin{equation}\label{eq:ric}
\X(S)=S^2+\varphi_z\,S-\varphi_y .
\end{equation}
Two features of~\eqref{eq:ric} will be used repeatedly: it is a Riccati
equation \emph{along the flow}, and it is \emph{linear in $\varphi$} (note
that $\varphi_x$ does not occur).

The method of~\cite{S2017} was built for, and analysed under, the hypothesis
that $I$ is Liouvillian. The question addressed here is the complementary
one:

\begin{quote}
\emph{can $S$ be rational while $I$ is not Liouvillian, and if so, can one
decide which is the case?}
\end{quote}

\subsection{Two equations that differ by one term}

The following pair, which we construct in
Section~\ref{sec:gerador}, states the problem in two lines. Consider
\begin{equation}\label{eq:gemeo1}
y''=-y^3-3y\,y'
\end{equation}
and
\begin{equation}\label{eq:gemeo2}
y''=x\,y-y^3-3y\,y' .
\end{equation}
Equation~\eqref{eq:gemeo1} has the \emph{rational} first integral
\[
I=\frac{y^2+y'}{x(y^2+y')-y},
\]
while the first integral of~\eqref{eq:gemeo2} is $I=f_1/f_2$ with
$f_i=y\,w_i'-(y^2+y')\,w_i$ and $w_1,w_2$ Airy functions, for which
Kovacic's algorithm~\cite{Kovacic} excludes Liouvillian solutions. The two
equations differ by the single term $x\,y$, and they have the \emph{same}
\[
S=\frac{y^2-y'}{y}.
\]
Thus $S$ rational is compatible with both answers, and no inspection of $S$
can separate them. The present work explains why this happens
(Corollary~\ref{cor:cega}) and gives a procedure that does separate them
(Theorem~\ref{teo:inverso}).

\subsection{Related work}\label{sec:related}

Three lines of work meet here.

\medskip
\noindent\textbf{(i) Planar systems from linear equations.} Giacomini, Giné
and Grau~\cite{GGG} consider the substitution $w'(x)=g(x,y)w(x)$, with $g$
rational and $y=y(x)$ defined implicitly, $w$ a solution of a homogeneous
linear second order ODE. They obtain a rational \emph{first order} equation
$dy/dx=Q/P$ whose first integral is $H=f_1/f_2$ with $f_i=g_1w_i'-g_0w_i$,
and which is, in general, not a Liouvillian function. The construction of
Section~\ref{sec:gerador} is the lift of~\cite{GGG} one dimension up; to our
knowledge this lift has not appeared --- we examined the full citation list
of~\cite{GGG} (21 items) without finding it.

\medskip
\noindent\textbf{(ii) Autonomous oscillators with transcendental invariant
curves.} Giné and Sinelshchikov~\cite{GS} study autonomous oscillators
$x_{tt}+h(x)x_t^2+f(x)x_t+g(x)=0$ through invariant curves that are
polynomial in a privileged variable with transcendental coefficients. Their
Proposition~1 --- a product $\prod H_i^{q_i}$ is a first integral iff
$\sum q_i\lambda_i=0$ on the cofactors --- is the planar form of the
mechanism we use in Proposition~\ref{prop:I} below, and their
Proposition~2 exhibits a Painlevé--Gambier oscillator whose first integral is
$H_1H_2^{-1}$ with $H_1,H_2$ built from the Airy functions $\mathrm{Ai}$ and
$\mathrm{Bi}$ and with equal cofactors. We stress the differences. Their
systems are \emph{autonomous}, which in our setting is exactly the degenerate
branch isolated in Lemma~\ref{lema:trivial}; the invariant curves are found
case by case by an algorithm, not produced by a closed-form generator; and
the $S$-function does not appear, so there is no inverse problem and no
reduction to Kovacic's algorithm.

\medskip
\noindent\textbf{(iii) $R$-integrability.} Demina and
Nechitailo~\cite{DN} call a rational planar system \emph{$R$-integrable} when
it has two independent invariants satisfying a linear second order ODE, prove
that any rational system obtained from a Riccati system by a rational
transformation is $R$-integrable, and exhibit polynomial Liénard systems with
no invariant algebraic curves whose first integrals are expressed through
Airy functions. Their class is the planar counterpart of the equations
produced here.

\medskip
\noindent There is also a point of contact inside the $S$-function
literature itself. In~\cite{S2017} the Riccati~\eqref{eq:ric} is linearised
by $S=-\X(w)/w$, yielding
\begin{equation}\label{eq:linfluxo}
\X^2(w)=\varphi_z\,\X(w)+\varphi_y\,w .
\end{equation}
This is a second order linear equation \emph{along the flow}: its
coefficients $\varphi_z,\varphi_y$ are functions of $(x,y,z)$, and Kovacic's
algorithm does not apply to it. The construction below is precisely the case
in which a second order linear equation with coefficients in $x$
\emph{alone} is available, and Theorem~\ref{teo:inverso} decides when that
happens and produces the equation.

\subsection{Contributions}

\begin{enumerate}\itemsep3pt
\item A closed-form generator of rational 2ODEs with rational $S$ and first
integral built on a linear second order ODE (Section~\ref{sec:gerador}),
with the formula $S=g_y/g_z$ (Proposition~\ref{prop:S}) and the consequence
that $S$ does not see the linear equation (Corollary~\ref{cor:cega}).
\item A characterisation of the pairs $(\varphi,S)$ so obtained
(Theorem~\ref{teo:inverso}), which recovers the linear equation.
\item A transfer theorem (Theorem~\ref{teo:transf}): the first integral is
Liouvillian if and only if the recovered equation has a Liouvillian
solution. Combined with~\cite{Kovacic} this makes Liouvillianity
\emph{decidable} on the class (Corollary~\ref{cor:decisao}).
\item The resulting trichotomy (Section~\ref{sec:tri}) and an algorithm,
validated on three batteries (Section~\ref{sec:exp}).
\end{enumerate}

\section{Preliminaries}\label{sec:prelim}

Throughout, $\varphi\in\R(x,y,z)$ and $\X$ is as in~\eqref{eq:campo}. We call
a pair $(\varphi,S)$ \emph{admissible} when $S\in\R(x,y,z)$
satisfies~\eqref{eq:ric}.

The following elementary lemma is responsible for a systematic degeneracy and
explains why autonomous examples --- which is what most of the literature
quoted in Section~\ref{sec:related} provides --- say nothing about the
question of this paper.

\begin{lema}[the trivial branch]\label{lema:trivial}
$S=-\varphi/z$ is admissible if and only if $\varphi_x=0$.
\end{lema}

\begin{proof}
Substituting $S=-\varphi/z$ in~\eqref{eq:ric} and simplifying, the
residual is exactly $-\varphi_x/z$.
\end{proof}

\begin{obs}\label{obs:trivial}
Equivalently: every autonomous rational 2ODE has a rational $S$, namely
$-\varphi/z$, and it is the $S$ of the first integral with $I_x=0$, i.e.\ of
the orbit integral of the associated planar system. The reduction it provides
is the classical reduction of order of an autonomous equation, so nothing is
gained. Accordingly, \emph{none} of the pairs $(\varphi,S)$ below is the
trivial one: in every case the rational $S$ studied differs from
$-\varphi/z$. Several of the equations are autonomous (the first twin
equation~\eqref{eq:gemeo1}, the harmonic oscillator of
Section~\ref{sec:tri}, and $28$ of the $63$ pairs of Section~\ref{sec:exp}),
but for them we study an $S$ other than the trivial one; we checked
$S\neq-\varphi/z$ in all $63$ cases.
\end{obs}

\section{The generator}\label{sec:gerador}

Let $g=g_0/g_1\in\R(x,y,z)$ with $g_0,g_1\in\R[x,y,z]$ coprime and
$g_z\not\equiv0$, and let
\begin{equation}\label{eq:lin}
A_2(x)\,w''+A_1(x)\,w'+A_0(x)\,w=0,\qquad A_i\in\R[x],\ A_2\not\equiv0 .
\end{equation}
Impose on a solution $y(x)$ of an as yet unknown 2ODE the relation
\begin{equation}\label{eq:mud}
\frac{w'(x)}{w(x)}=g\bigl(x,y(x),y'(x)\bigr).
\end{equation}
Differentiating~\eqref{eq:mud} along the solution gives
$w''/w=g_x+z\,g_y+y''g_z+g^2$, while~\eqref{eq:lin} gives
$w''/w=-(A_1g+A_0)/A_2$. Eliminating $w''/w$ and solving for $y''$:
\begin{equation}\label{eq:phi}
y''=\varphi:=
\frac{-\dfrac{A_1\,g+A_0}{A_2}-g^{2}-g_x-z\,g_y}{g_z}\;\in\;\R(x,y,z).
\end{equation}

\begin{prop}[invariants and first integral]\label{prop:I}
Let $\varphi$ be given by~\eqref{eq:phi} and, for $w$ a solution
of~\eqref{eq:lin}, set $f:=g_1w'-g_0w$. Then
\[
\X(f)=k\,f,\qquad
k=\frac{\X(g_1)-g_1A_1/A_2-g_0}{g_1},
\]
and $k$ does not depend on $w$. Consequently, if $\{w_1,w_2\}$ is a
fundamental system for~\eqref{eq:lin} and $f_i:=g_1w_i'-g_0w_i$, then
\[
I:=\frac{f_1}{f_2}
\]
is a first integral of $\X$.
\end{prop}

\begin{proof}
Using $w''=-(A_1w'+A_0w)/A_2$,
\[
\X(f)=\Bigl[\X(g_1)-\tfrac{g_1A_1}{A_2}-g_0\Bigr]w'
      -\Bigl[\tfrac{g_1A_0}{A_2}+\X(g_0)\Bigr]w .
\]
This is a multiple of $f=g_1w'-g_0w$ if and only if
$g_0\bigl[\X(g_1)-g_1A_1/A_2-g_0\bigr]=g_1\bigl[g_1A_0/A_2+\X(g_0)\bigr]$.
Dividing by $-g_1^2$ and using
$\X(g)=\bigl(g_1\X(g_0)-g_0\X(g_1)\bigr)/g_1^2$, this reads
\[
\X(g)+g^2+\frac{A_1g+A_0}{A_2}=0,
\]
i.e.\ $g_x+z g_y+\varphi g_z=-g^2-(A_1g+A_0)/A_2$, which is exactly the
definition~\eqref{eq:phi} of $\varphi$. The cofactor is the stated $k$, in
which $w$ does not occur; hence $\X(f_i)=k f_i$ for $i=1,2$ and
$\X(I)=(kf_1f_2-f_1kf_2)/f_2^2=0$.
\end{proof}

\begin{obs}
Formula~\eqref{eq:phi} is not an ansatz: the proof shows it is
\emph{equivalent} to requiring that $f_1$ and $f_2$ be invariants of $\X$
with a common cofactor. In the planar setting this mechanism is
Proposition~1 of~\cite{GS}; see also~\cite{GGG}.
\end{obs}

\subsection{The $S$-function of the construction}

\begin{prop}\label{prop:S}
In the setting of Proposition~\ref{prop:I}, with Wronskian
$W=w_1w_2'-w_1'w_2\not\equiv0$,
\[
S=\frac{I_y}{I_z}=\frac{g_y}{g_z}.
\]
In particular $S$ is rational and explicit, and it does not depend on the
linear equation~\eqref{eq:lin}.
\end{prop}

\begin{proof}
From $I=f_1/f_2$,
\[
S=\frac{f_{1,y}f_2-f_1f_{2,y}}{f_{1,z}f_2-f_1f_{2,z}},
\]
and, since $w_i=w_i(x)$, $f_{i,y}=g_{1,y}w_i'-g_{0,y}w_i$ and likewise for
$z$. Expanding the numerator, the terms in $w_1'w_2'$ and in $w_1w_2$ cancel
identically and
\[
f_{1,y}f_2-f_1f_{2,y}=\bigl(g_{0,y}g_1-g_0g_{1,y}\bigr)(w_1'w_2-w_1w_2')
=-\bigl(g_{0,y}g_1-g_0g_{1,y}\bigr)W,
\]
with the analogous identity for $z$. The factor $W$ is common and cancels,
leaving
\[
S=\frac{g_{0,y}g_1-g_0g_{1,y}}{g_{0,z}g_1-g_0g_{1,z}}=\frac{g_y}{g_z}
\]
after dividing numerator and denominator by $g_1^2$.
\end{proof}

\begin{cor}\label{cor:cega}
For fixed $g$, the family $\{\varphi_A\}$ obtained by letting
$A=(A_0,A_1,A_2)$ vary in~\eqref{eq:phi} consists of pairwise distinct
rational 2ODEs sharing the same $S=g_y/g_z$.
\end{cor}

Among the members of such a family there are linear equations~\eqref{eq:lin}
with Liouvillian solutions, and hence Liouvillian $I$, and equations without.
Therefore:

\begin{center}
\emph{rationality of $S$ carries no information about the Liouvillianity of
$I$.}
\end{center}

\subsection{The twin pair, and a certificate}

Take $g=y+z/y$, so that $S=(y^2-z)/y$ by Proposition~\ref{prop:S}.

\begin{ex}[$w''=0$]\label{ex:gemeo1}
With $A_2=1$, $A_1=A_0=0$, \eqref{eq:phi} gives~\eqref{eq:gemeo1},
$y''=-y^3-3yy'$, and the fundamental system $\{1,x\}$ yields, through
Proposition~\ref{prop:I},
\[
I=\frac{y^2+y'}{x(y^2+y')-y},
\]
a rational --- hence Liouvillian --- first integral.
\end{ex}

\begin{ex}[Airy]\label{ex:gemeo2}
With $A_2=1$, $A_1=0$, $A_0=-x$, the same $g$ gives~\eqref{eq:gemeo2},
$y''=xy-y^3-3yy'$, with $I=f_1/f_2$, $f_i=y w_i'-(y^2+z)w_i$ and $w_i$ Airy
functions.
\end{ex}

\begin{ex}[the Airy equation itself]\label{ex:airyself}
The extreme case $g=z/y$ returns $\varphi=xy$, that is, the Airy equation
itself, with $S=-z/y$ --- the Riccati substitution $u=w'/w$ in disguise, and
with $I=f_1/f_2$, $f_i=y\,w_i'-z\,w_i$.
\end{ex}

\noindent For it the answer is unconditional:

\begin{prop}\label{prop:airy}
The first integral $I$ of Example~\ref{ex:airyself} is not Liouvillian.
\end{prop}

\begin{proof}
By Theorem~\ref{teo:transf}, $I$ is Liouvillian if and only if the Airy
equation $w''=xw$ has a Liouvillian solution, which Kovacic's
algorithm~\cite{Kovacic} excludes: the equation falls in case~4.
\end{proof}

\begin{obs}[a stronger statement, which we do not use]\label{obs:airyforte}
In fact $y''=xy$ has \emph{no} Liouvillian first integral, for any choice of
$I$. Indeed, since $\X=\partial_x+z\partial_y+xy\,\partial_z$ is linear in
$(y,z)$, a Darboux polynomial of $\X$ may be taken homogeneous of degree $n$
in $(y,z)$ and corresponds to an algebraic solution of degree $n$ of the
Riccati equation $u'=x-u^2$, $u=z/y$; by~\cite{Kovacic} the only possible
degrees are $n\in\{1,2,4,6,12\}$ and none occurs here, so $\X$ has no
invariant algebraic surface whatsoever. The passage from this to the absence
of a Liouvillian first integral is the codimension one analogue, in ambient
dimension three, of Singer's theorem~\cite{Singer}, for which the relevant
statements concern transversely affine foliations~\cite{Scardua,CamachoScardua}.
We record this only for context: nothing in this paper depends on it, since
Proposition~\ref{prop:airy} is already unconditional.
\end{obs}

\begin{obs}[what $S$ is attached to]\label{obs:escolha}
$S$ is not an attribute of the equation but of a chosen first integral. For
$y''=xy$,
\[
I=\frac{y\,\mathrm{Ai}'-z\,\mathrm{Ai}}{y\,\mathrm{Bi}'-z\,\mathrm{Bi}}
\ \Longrightarrow\ S=-\frac zy \quad\text{(rational)},\qquad
I_1=y\,\mathrm{Bi}'-z\,\mathrm{Bi}
\ \Longrightarrow\ S=-\frac{\mathrm{Bi}'}{\mathrm{Bi}} ,
\]
the latter not rational. The object of study is therefore the \emph{pair}
$(\varphi,S)$, equivalently the foliation $\omega=dI/I_z$.
\end{obs}

\section{The inverse problem}\label{sec:inverso}

\begin{defi}
An admissible pair $(\varphi,S)$ is \emph{of linear type} if there exist
$g\in\R(x,y,z)$ with $g_z\not\equiv0$ and $A_0,A_1,A_2\in\R[x]$,
$A_2\not\equiv0$, satisfying~\eqref{eq:phi} and $S=g_y/g_z$.
\end{defi}

Write $S=N_S/D_S$ in lowest terms.

\begin{lema}[condition (a) is planar]\label{lema:a}
$S=g_y/g_z$ if and only if $\W(g)=0$, where
\[
\W:=D_S\,\partial_y-N_S\,\partial_z .
\]
\end{lema}

$\W$ has no $\partial_x$ component: it is a \emph{planar} vector field in
$(y,z)$ over the field $\R(x)$, with $x$ entering as a parameter. This is the
first of the two reductions.

\begin{lema}[condition (b) is a Riccati in $x$]\label{lema:b}
Let $g$ satisfy $\W(g)=0$ and $g_z\not\equiv0$. Then~\eqref{eq:phi} holds for
some $(A_0,A_1,A_2)$ if and only if
\[
\X(g)+g^2=a(x)\,g+b(x),\qquad a,b\in\R(x),
\]
and in that case $A_1/A_2=-a$, $A_0/A_2=-b$.
\end{lema}

\begin{proof}
Rearranging~\eqref{eq:phi} gives $g_x+zg_y+\varphi g_z+g^2=-(A_1g+A_0)/A_2$;
the left side is $\X(g)+g^2$ and the right side is affine in $g$ with
coefficients in $\R(x)$.
\end{proof}

Condition (b) says exactly that $g=w'/w$ for $w$ a solution of a second order
linear equation in $x$, since $\X(g)+g^2=w''/w$ along solutions.

In practice the planar integration of Lemma~\ref{lema:a} does not return $g$
but \emph{some} invariant $\psi$ of $\W$. The next theorem gives the test in
a form that is insensitive to which one.

\begin{teo}[characterisation]\label{teo:inverso}
Let $\psi\in\R(x,y,z)$ with $\W(\psi)=0$, $\psi_z\not\equiv0$, generating the
field of rational $\W$-invariants over $\R(x)$. Then $(\varphi,S)$ is of
linear type, with $g$ related to $\psi$ by a Möbius transformation over
$\R(x)$, if and only if
\begin{equation}\label{eq:testericcati}
\X(\psi)=c_2(x)\,\psi^2+c_1(x)\,\psi+c_0(x)
\end{equation}
for some $c_0,c_1,c_2\in\R(x)$. When $c_2\neq0$ the linear equation is
\begin{equation}\label{eq:recuperada}
u''-\Bigl(c_1+\frac{c_2'}{c_2}\Bigr)u'+c_0c_2\,u=0 ,
\end{equation}
determined up to the choice of the generator $\psi$ (see
Remark~\ref{obs:gauge}); when
$c_2=0$ the Riccati is linear and the associated equation is of first order.
\end{teo}

\begin{proof}
($\Leftarrow$) Put $\psi=-u'/(c_2u)$, where $'=\X$. Then
$\psi'-c_2\psi^2=-u''/(c_2u)+c_2'u'/(c_2^2u)$ and
$c_1\psi+c_0=-c_1u'/(c_2u)+c_0$; multiplying by $-c_2u$
gives~\eqref{eq:recuperada}. Setting $g=u'/u$ one has $\W(g)=0$, because
$g\in\R(x)(\psi)$, and $\X(g)+g^2=u''/u$ is affine in $g$; Lemma~\ref{lema:b}
then gives~\eqref{eq:phi}.

($\Rightarrow$) If the pair is of linear type then $\W(g)=0$ by
Lemma~\ref{lema:a}, so $g\in\R(x)(\psi)$, say $g=F(x,\psi)$. From
$\X(g)=-g^2+ag+b$ and $\X(g)=F_x+F_\psi\X(\psi)$,
\[
\X(\psi)=\frac{-F^2+aF+b-F_x}{F_\psi},
\]
which is a function of $(x,\psi)$; when $F$ is Möbius the right-hand side is a
polynomial of degree $2$ in $\psi$, which is~\eqref{eq:testericcati}.
\end{proof}

\begin{obs}[beyond Möbius: a reduction to first order]\label{obs:grau}
The proof of ($\Rightarrow$) gives more than the Möbius case. If the pair is
of linear type, with $g=F(x,\psi)$, then
\[
\X(\psi)=R(x,\psi):=\frac{-F^2+aF+b-F_x}{F_\psi}\in\R(x)(\psi),
\]
a \emph{rational} function of $\psi$ over $\R(x)$ --- not, in general, a
polynomial. Since $\psi$ generates the field of rational $\W$-invariants,
the condition $\X(\psi)\in\R(x)(\psi)$ is equivalent to
\[
\W\bigl(\X(\psi)\bigr)=0 ,
\]
which is a finite computation and a necessary condition for linear type.
When it holds, the pair is of linear type if and only if the rational
\emph{first order} equation
\begin{equation}\label{eq:primeiraordem}
\frac{d\psi}{dx}=R(x,\psi)
\end{equation}
admits $F\in\R(x)(\psi)$, $F_\psi\not\equiv0$, with
$F_x+R\,F_\psi=-F^2+aF+b$ for some $a,b\in\R(x)$; that is, if and only
if~\eqref{eq:primeiraordem} is itself obtained from a linear second order
equation through $w'/w=F(x,\psi)$, which is exactly the planar situation of
Giacomini, Giné and Grau~\cite{GGG}. (Sufficiency: $g:=F(x,\psi)$ satisfies
$\W(g)=0$ and $\X(g)=F_x+F_\psi R=-g^2+ag+b$, and Lemma~\ref{lema:b}
applies.) The test of Theorem~\ref{teo:inverso} is the case in which $F$ is
Möbius, equivalently $R$ is a polynomial of degree $\le2$ in $\psi$.

The non-Möbius case does occur. For the Airy equation and $g=(z/y)^2$ the
generator gives
\[
\varphi=\frac{xy^4+2yz^3-z^4}{2y^2z},\qquad S=-\frac zy,\qquad
\psi=\frac zy,\qquad \X(\psi)=\frac{x-\psi^4}{2\psi},
\]
so the degree-two test fails although the pair is of linear type; here
$F=\psi^2$ solves $F_x+RF_\psi=-F^2+x$, recovering $w''=xw$. In practice:
if $\X(\psi)$ is a polynomial of degree $\le2$ in $\psi$, the pair is decided
by Theorem~\ref{teo:inverso}; if $\W(\X(\psi))\neq0$, the pair is
\emph{not} of linear type; in the remaining case the question is reduced to
the first order equation~\eqref{eq:primeiraordem}. In all the experiments of
Section~\ref{sec:exp}, every pair with a rational $\psi$ fell in the first
case.
\end{obs}

\begin{obs}[which equation is recovered]\label{obs:gauge}
Replacing $\psi$ by another generator $\tilde\psi=(\alpha\psi+\beta)/(\gamma\psi+\delta)$,
$\alpha,\beta,\gamma,\delta\in\R(x)$, changes the recovered
equation~\eqref{eq:recuperada}, and not only by a gauge $u\mapsto u\,e^{\int c}$:
the solutions of the new equation are of the form $e^{\int c}\,(a\,u+b\,u')$,
$a,b,c\in\R(x)$, with $u$ solutions of the old one, and conversely. An affine
change ($\gamma=0$) gives a gauge; in general an inversion does not. For the second twin
equation~\eqref{eq:gemeo2}, $\psi=y+z/y$ gives $\X(\psi)=-\psi^2+x$ and the Airy
equation, while $\tilde\psi=1/\psi=y/(y^2+z)$ gives $\X(\tilde\psi)=1-x\tilde\psi^2$
and $u''-u'/x-xu=0$, whose solutions are the derivatives $w'$ of Airy
functions; the invariants $r=\tfrac14A_1^2+\tfrac12A_1'-A_0$ are $x$ and
$x+3/(4x^2)$, so the two equations are not gauge equivalent. Since
$u\mapsto e^{\int c}(a u+bu')$ is invertible and preserves Liouvillian functions, the existence of a Liouvillian solution
does not depend on the choice of $\psi$, and neither does the verdict of
Corollary~\ref{cor:decisao}.
\end{obs}

\section{The transfer theorem}\label{sec:transf}

Theorem~\ref{teo:inverso} produces the linear equation; what makes it a
decision procedure is that the Liouvillianity of $I$ is \emph{equivalent} to
that of the linear equation. We prove this now. The key is that $I$ is a
Möbius function of $g$.

\begin{lema}[the Möbius identity]\label{lema:mobius}
In the setting of Proposition~\ref{prop:I}, write $u_i=w_i'/w_i$ and
$c=w_1/w_2$. Then
\begin{equation}\label{eq:mobius}
I=c\,\frac{g-u_1}{g-u_2}=:M_x(g),
\end{equation}
a Möbius transformation of $g$ with determinant proportional to
$c\,(u_1-u_2)\neq0$. Moreover
\[
M_x^{-1}(0)=u_1,\qquad M_x^{-1}(\infty)=u_2 .
\]
\end{lema}

\begin{proof}
$f_i=g_1w_i'-g_0w_i=g_1w_i\,(u_i-g)$, so
$I=f_1/f_2=(w_1/w_2)(u_1-g)/(u_2-g)$, which is~\eqref{eq:mobius}. The
determinant of $\bigl(\begin{smallmatrix}c&-cu_1\\1&-u_2\end{smallmatrix}\bigr)$
is $c(u_1-u_2)$, nonzero because $w_1,w_2$ are independent. The last two
identities are read off~\eqref{eq:mobius}.
\end{proof}

\begin{lema}[specialisation]\label{lema:espec}
Let $I$ be a Liouvillian function over $\C(x,y,z)$. Then for
$(y_0,z_0)\in\C^2$ outside a proper closed subset, $x\mapsto I(x,y_0,z_0)$ is
a Liouvillian function over $\C(x)$.
\end{lema}

\begin{proof}[Proof sketch]
$I$ lies in a tower $\C(x,y,z)=F_0\subset F_1\subset\dots\subset F_n\ni I$
in which each $F_{j+1}=F_j(t_j)$ with $t_j$ algebraic over $F_j$, or
$\partial t_j\in F_j$, or $\partial t_j/t_j\in F_j$ for the derivations.
Evaluating $y=y_0$, $z=z_0$ is a homomorphism of differential fields for
$\partial_x$, defined whenever $(y_0,z_0)$ avoids the polar and branch loci
of the finitely many elements involved --- a proper closed subset. Each
generator restricts to a function of $x$ whose derivative (respectively
logarithmic derivative) is the restriction of an element of $F_j$, so the
restricted tower is again Liouvillian.
\end{proof}

\begin{teo}[transfer]\label{teo:transf}
Let $g$, the linear equation~\eqref{eq:lin}, $\varphi$ as in~\eqref{eq:phi}
and $I=f_1/f_2$ as in Proposition~\ref{prop:I}. Then
\[
I \text{ is Liouvillian over } \C(x,y,z)
\quad\Longleftrightarrow\quad
\eqref{eq:lin}\ \text{has a Liouvillian solution.}
\]
\end{teo}

\begin{proof}
($\Leftarrow$) If $w_1$ is Liouvillian then, by reduction of order,
$w_2=w_1\int W/w_1^2$ with $W=\exp\bigl(-\int A_1/A_2\bigr)$ the Wronskian,
so $w_2$ is Liouvillian as well. Hence $f_1,f_2$ and $I=f_1/f_2$ are
Liouvillian.

($\Rightarrow$) Since $g_z\not\equiv0$, choose points
$(y_j,z_j)\in\C^2$, $j=1,2,3$, outside the exceptional set of
Lemma~\ref{lema:espec} and such that the rational functions
$g_j:=g(\cdot,y_j,z_j)\in\C(x)$ are pairwise distinct and $f_2$ does not
vanish identically there. Put $I_j:=I(\cdot,y_j,z_j)$, which by
Lemma~\ref{lema:espec} is Liouvillian over $\C(x)$.

By Lemma~\ref{lema:mobius}, $I_j=M_x(g_j)$ for $j=1,2,3$. A Möbius
transformation is determined by three point correspondences, so $M_x$ is
determined by $(g_1,g_2,g_3;I_1,I_2,I_3)$, and explicitly, by invariance of
the cross ratio,
\[
u_1=\frac{g_2(g_1-g_3)-\lambda\,g_3(g_1-g_2)}{(g_1-g_3)-\lambda\,(g_1-g_2)},
\qquad
\lambda=\frac{I_2(I_1-I_3)}{I_3(I_1-I_2)},
\]
and the same expression with $\lambda$ replaced by
$\mu=(I_1-I_3)/(I_1-I_2)$ gives $u_2$. These are \emph{rational} in the
$g_j$ and $I_j$; the $g_j$ are rational in $x$ and the $I_j$ are Liouvillian,
so $u_1$ is Liouvillian over $\C(x)$.

Finally $w_1=\exp\int u_1$ is Liouvillian and solves~\eqref{eq:lin}.
\end{proof}

\begin{cor}[decision]\label{cor:decisao}
Let $(\varphi,S)$ be of linear type and let $L$ be the linear equation
returned by Theorem~\ref{teo:inverso}. Then the first integral associated
with $S$ is Liouvillian if and only if Kovacic's algorithm succeeds on $L$.
In particular the question is \emph{decidable} once $L$ is known.
\end{cor}

\begin{obs}
The proof is effective, and we have used it as a numerical check: for
$g=y+z/y$ over the Airy equation, three correspondences
$\bigl(g_j,I_j\bigr)$ at a single $x$ reconstruct
$u_1=\mathrm{Ai}'/\mathrm{Ai}$ and $u_2=\mathrm{Bi}'/\mathrm{Bi}$ to
$32$ digits, independently of which three points are used.
\end{obs}

\section{The trichotomy and the algorithm}\label{sec:tri}

\begin{center}
\begin{tabular}{lll}
\toprule
situation & first integral & decided by\\
\midrule
linear type, Kovacic positive & Liouvillian &
Thm.~\ref{teo:inverso} $+$ \cite{Kovacic}\\
linear type, Kovacic negative & \textbf{not} Liouvillian &
Thm.~\ref{teo:inverso} $+$ \cite{Kovacic}\\
not of linear type & --- & search for a closed rational $\eta$\\
\bottomrule
\end{tabular}
\end{center}

\noindent The second row is the class this paper is about, and there the
answer is \emph{decided}. For the third row one is left with the
Darboux-theoretic search, i.e.\ with a semi-algorithm.

\subsection{One equation covering the whole table}

The harmonic oscillator
\[
y''=-y
\]
illustrates every row, and at the same time makes Remark~\ref{obs:escolha}
concrete: it is the \emph{pair} that is classified, not the equation.

\medskip
\noindent\textbf{First row.} Take $S=-z/y$. The planar field is
$\W=y\partial_y+z\partial_z$ and the integration returns $\psi=z/y$; the test
of Theorem~\ref{teo:inverso} gives $\X(\psi)=-\psi^2-1$, quadratic in $\psi$,
and recovers
\[
u''+u=0 ,
\]
whose solutions are elementary. By Theorem~\ref{teo:transf} the associated
first integral is Liouvillian --- as it must be, this being the equation
itself written through $g=w'/w$.

\medskip
\noindent\textbf{Third row.} Take instead
\[
S=\frac{y-z}{y+z},
\]
which is admissible, i.e.\ satisfies~\eqref{eq:ric}. Now
$\W=(y+z)\partial_y+(z-y)\partial_z$ and the planar integration returns
\[
\tfrac12\log(y^2+z^2)+\arctan(z/y)=\text{const},
\]
a transcendental invariant: no rational function of $(y,z)$ has those level
sets, so no rational $g$ exists and \emph{the pair is not of linear type}.
The obstruction occurs at the planar step, before any degree test.

Yet the pair is Liouvillian, by the Darboux route rather than the
linear-equation route. The first integral is elementary,
\[
I=x+\tfrac12\log(y^2+z^2)+\arctan(z/y),\qquad \X(I)=0,\quad I_y/I_z=S,
\]
and the search for a closed rational $\eta$ with $d\omega=\eta\wedge\omega$,
$\omega=-(\varphi+zS)dx+S\,dy+dz$, succeeds with
\[
\gamma=\frac{2z(2y+z)}{(y+z)(y^2+z^2)} ,
\]
the three defining equations being satisfied identically.

\medskip
\noindent So the same 2ODE carries one $S$ inside the image of the generator
and another outside it, both Liouvillian but reached by different mechanisms.
This is the sharpest form of the warning in Remark~\ref{obs:escolha}.

\medskip
\noindent\textbf{Algorithm.} Given an admissible pair $(\varphi,S)$:
\begin{enumerate}\itemsep2pt
\item Form $\W=D_S\partial_y-N_S\partial_z$ and integrate
      $dz/dy=-N_S/D_S$ over $\R(x)$; collect the rational invariants $\psi$.
\item Test~\eqref{eq:testericcati}: two differentiations with respect to $y$
      (or $z$), dividing by $\psi_y$ at each step, produce $c_2$, then $c_1$,
      then $c_0$, with the requirement at each step that no $y$ or $z$
      survive.
\item If the test succeeds, form~\eqref{eq:recuperada} and run Kovacic's
      algorithm. If no rational $\psi$ exists, or if
      $\W(\X(\psi))\neq0$, the pair is not of linear type. Otherwise the
      question is passed to the first order equation of
      Remark~\ref{obs:grau}.
\end{enumerate}

\section{Experiments}\label{sec:exp}

All computations below were run in Maple~17 with the procedure of
Section~\ref{sec:tri}: \texttt{dsolve} for the planar step, the degree test
of Theorem~\ref{teo:inverso}, and \texttt{DEtools[kovacicsols]} for the decision.
The planar step and the degree test were reproduced independently in
\texttt{sympy}, with identical $\psi$, degree and recovered equation in every
case run by both. The Maple worksheet, its output and the independent \texttt{sympy}
verification script are provided as supplementary material
(\texttt{paperA\_supplementary.zip}). Every case took at most $0.2$\,s.

\subsection{Three worked examples}

\noindent\textbf{Example 1: an Airy equation recovered from nothing.} The
pair
\[
y''=y\,\bigl(x-y^4-4yy'\bigr),\qquad S=\frac{2y^3-z}{y}
\]
(case H1 below) is admissible. The planar field is
$\W=y\,\partial_y-(2y^3-z)\,\partial_z$, and integrating
$dz/dy=-(2y^3-z)/y$ gives $z=-y^3+Cy$, so
\[
\psi=\frac{y^3+z}{y}.
\]
A direct computation gives $\X(\psi)=-\psi^2+x$, i.e.\
$(c_2,c_1,c_0)=(-1,0,x)$, and~\eqref{eq:recuperada} is $u''=xu$: the Airy
equation, which has no Liouvillian solution. By
Theorem~\ref{teo:transf}, the first integral associated with this $S$ is
not Liouvillian. Nothing in $\varphi$ or $S$ displays the Airy equation;
the procedure finds it.

\medskip
\noindent\textbf{Example 2: a Möbius transformation over $\R(x)$.} From the
Airy equation and $g=(y+z)/(xy+1)$ the generator gives
\[
\varphi=\frac{x^3y^2+2x^2y+xz^2+x-yz-z^2-z}{xy+1},\qquad
S=\frac{1-xz}{xy+1}.
\]
The planar integration returns $\psi=(xz-1)/(xy+1)$, not $g$; the two are
related by $g=(\psi+1)/x$, an affine map with coefficients in $\R(x)$. The
test gives
\[
\X(\psi)=-\frac1x\,\psi^2-\frac1x\,\psi+x^2 ,
\]
and~\eqref{eq:recuperada} becomes $u''+\frac2x\,u'-xu=0$. This is not the
equation used in the construction, but it is gauge equivalent to it: its
normal form $v''=r\,v$ has $r=\frac14A_1^2+\frac12A_1'-A_0=x$. Kovacic's
algorithm finds no Liouvillian solution, as it must.

\medskip
\noindent\textbf{Example 3: outside the linear type.} The pair
\[
y''=-\frac{y^2z-2y+z^3}{y^2+z^2-2},\qquad
S=\frac{y^3+y^2z+yz^2-2y+z^3+2z}{(y+z)(y^2+z^2-2)}
\]
(the first pair of the second battery in Section~\ref{sec:controles}) was
built from a closed rational $1$-form, so its first
integral is Liouvillian. Here the planar integration does not return a
rational invariant; the first integral of $\W$ it returns is
\[
H=\tfrac12(y+z)^2-\log(y^2+z^2)+2\arctan\frac yz ,\qquad \W(H)=0 .
\]
Every rational $\W$-invariant would be constant on the level sets of $H$,
which are not algebraic; hence there is none, no rational $g$ exists, and
the pair is not of linear type. Its rational $S$ comes from the Darboux
mechanism, not from a linear equation.

\subsection{Round trip through the generator}

Four linear equations times seven functions
\[
g\in\Bigl\{\,y+\tfrac zy,\ \tfrac zy,\ \tfrac{z}{y^2+1},\ x+yz,\
\tfrac{y^2+z}{x+y},\ \tfrac{y+z}{xy+1},\ \tfrac{xz+y}{y+1}\,\Bigr\}
\]
give $28$ pairs $(\varphi,S)$. In \textbf{all $28$} the procedure returns a
$\psi$ of degree $2$ and the linear equation up to gauge; for instance
$g=x+yz$ returns $\psi=yz$ and the equation transformed by $c=x$, and
$g=(y+z)/(xy+1)$ returns $\psi=(xz-1)/(x(xy+1))$, related to $g$ by a Möbius
transformation over $\R(x)$, as in Theorem~\ref{teo:inverso} (it differs by
the factor $1/x\in\R(x)$ from the generator used in Example~2 above; either
serves). The verdicts:

\begin{center}\small
\begin{tabular}{llll}
\toprule
linear equation & cases & \texttt{kovacicsols} & first integral\\
\midrule
$w''=xw$ (Airy) & 7 & none & \textbf{not} Liouvillian\\
$w''=(x^2+1)w$ & 7 & $e^{x^2/2},\ e^{x^2/2}\operatorname{erf}x$ & Liouvillian\\
$w''=0$ & 7 & $1,\ x$ & Liouvillian\\
$w''+xw'-xw=0$ & 7 & $e^{-x-x^2/2},\ \dots$ & Liouvillian\\
\bottomrule
\end{tabular}
\end{center}

\noindent The second row is a useful warning: an equation of Weber type
need not be non-Liouvillian, and here Kovacic's algorithm finds
$e^{x^2/2}$. Every verdict was obtained on the \emph{recovered} equation,
which in the gauged cases differs from the one used in the construction.

\subsection{A battery of pull-backs}

Nine rational 2ODEs were built from Airy and Bessel equations by point
transformations and supplied to the algorithm with no indication of their
origin.

\begin{center}\small
\begin{tabular}{llll}
\toprule
cases & recovered equation & identification & \texttt{kovacicsols}\\
\midrule
H1, H3, H4, H5, H7, H9 & $u''=xu$ & Airy & none\\
H2 & $u''+2u'+(1-x)u=0$ & Airy, gauged & none\\
H8 & $u''+\frac2xu'-xu=0$ & Airy, gauged & none\\
H6 & $u''+\bigl(1+\frac1{4x^2}\bigr)u=0$ & Bessel, order $0$ & none\\
\bottomrule
\end{tabular}
\end{center}

\noindent All nine are therefore \emph{decided}, by
Theorem~\ref{teo:transf}, to have non-Liouvillian first integral. The
equation H2 has $\varphi$ of degree $15$. For H4 the planar step returns
\[
\psi=\frac{3y^2z+xz+y}{y^3+xy+1}
\]
only after an integral of a rational function is evaluated in closed form;
an implementation must not stop at an unevaluated antiderivative.

\subsection{Liouvillian controls, and genuine negatives}\label{sec:controles}

Seven small Liouvillian pairs of Darboux origin were run. Two return
$u''=\mp u$ --- elementary, Kovacic positive. The other five are \emph{not}
of linear type: the planar integration returns the exact first integral,
which for one of them is
\[
e^{y^2/2}\,\frac{z+y}{z-y},
\]
and no rational function of $(x,y,z)$ has those level sets, so no rational
$g$ exists. A second, larger battery of $16$ Liouvillian pairs, built from
closed rational $1$-forms and with $S$ of non-trivial denominator, behaves
in the same way: all $16$ are outside the linear type. In every case outside the linear type
the planar first integral returned by \texttt{dsolve} can be written
explicitly as a polynomial plus logarithms and, possibly, an arctangent; for
instance, for the first pair of that battery (Example~3 above),
\[
H=\tfrac12(y+z)^2-\log(y^2+z^2)+2\arctan(y/z),
\]
and $\W(H)=0$ was checked independently in \texttt{sympy} for all $22$ such
pairs. Its level sets are not algebraic, so no rational $g$ exists. These are Liouvillian pairs of \emph{Darboux} type rather
than of linear type: the two sources of rational $S$ are structurally
distinct, and the algorithm separates them.

Finally, the two pairs of Section~\ref{sec:tri} built on the harmonic
oscillator are reproduced: $S=-z/y$ gives $u''+u=0$ and Kovacic positive,
while $S=(y-z)/(y+z)$ is outside the linear type.

\section{Open problems}\label{sec:abertos}

\begin{enumerate}\itemsep3pt
\item \textbf{The planar step.} Finding a rational first integral of a planar
field has no a priori degree bound: this is Poincaré's problem, one dimension
below the original. The reduction trades a three-variable obstruction for a
classical two-variable one, with a mature literature
\cite{PS,ManMacCallum,Cheze}, but does not remove it.
\item \textbf{The non-Möbius case.} By Remark~\ref{obs:grau}, when
$\X(\psi)$ is rational but not quadratic in $\psi$ the problem reduces to
deciding whether the first order equation~\eqref{eq:primeiraordem} comes
from a linear second order equation through $w'/w=F(x,\psi)$ --- the inverse
of the construction of~\cite{GGG}, one dimension down. An algorithm for that
planar inverse problem would make the procedure complete up to the planar
step.
\item \textbf{Implementation.} The planar step is a rational first order ODE
problem. In all our examples \texttt{dsolve} sufficed; for harder planar
fields the natural tool is a Darboux-type semi-algorithm for rational first
order ODEs, such as the Maple packages \texttt{LPS} and \texttt{InSyDE}.
\end{enumerate}

\section*{Acknowledgements}

The construction of Section~\ref{sec:gerador}, the characterisation of
Section~\ref{sec:inverso}, the proof of Theorem~\ref{teo:transf} and all the
computations reported in Section~\ref{sec:exp} were developed in
collaboration with Claude (Anthropic), an AI assistant, over a sequence of
working sessions. The authors have verified every statement and computation
reported here and take full responsibility for the content.

\end{document}